\documentclass{amsart}

\usepackage{amsmath}
\usepackage{amsthm}
\usepackage{amsfonts}
\usepackage{amssymb}

\usepackage{graphics}
\usepackage{epsfig}
\usepackage{color}
\usepackage{verbatim}

\newcommand\R{{\mathbf{R}}}

\newcommand\Z{{\mathbf{Z}}}

\renewcommand\P{{\mathbf{P}}}
\newcommand\E{{\mathbf{E}}}

\renewcommand\Im{{\operatorname{Im}}}
\renewcommand\Re{{\operatorname{Re}}}
\newcommand\eps{{\varepsilon}}

\newcommand\Dyson{{\operatorname{Sine}}}
\renewcommand\th{{\operatorname{th}}}

\subjclass{15A52}

\theoremstyle{plain}
  \newtheorem{theorem}{Theorem}
  \newtheorem{conjecture}[theorem]{Conjecture}

  \newtheorem{lemma}[theorem]{Lemma}
  \newtheorem{corollary}[theorem]{Corollary}

\theoremstyle{definition}
  \newtheorem{definition}[theorem]{Definition}
  \newtheorem{example}[theorem]{Example}
  \newtheorem{remark}[theorem]{Remark}

\begin{document}

\title[The Wigner-Dyson-Mehta conjecture]{The Wigner-Dyson-Mehta bulk universality conjecture for Wigner matrices}

\author{Terence Tao}
\address{Department of Mathematics, UCLA, Los Angeles CA 90095-1555}
\email{tao@math.ucla.edu}
\thanks{T. Tao is supported by a grant from the MacArthur Foundation, by NSF grant DMS-0649473, and by the NSF Waterman award.}

\author{Van Vu}
\address{Department of Mathematics, Rutgers, Piscataway, NJ 08854}
\email{vanvu@math.rutgers.edu}
\thanks{V. Vu is supported by research grants DMS-0901216 and AFOSAR-FA-9550-09-1-0167.}

\begin{abstract}  A well known conjecture of Wigner, Dyson, and Mehta asserts that the (appropriately normalized) $k$-point correlation functions of the eigenvalues of random $n \times n$ Wigner matrices in the bulk of the spectrum converge (in various senses) to the $k$-point correlation function of the Dyson sine process in the asymptotic limit $n \to \infty$.  There has been much recent progress on this conjecture; in particular, it has been established under a wide variety of decay, regularity, and moment hypotheses on the underlying atom distribution of the Wigner ensemble, and using various notions of convergence.  Building upon these previous results, we establish new instances of this conjecture with weaker hypotheses on the atom distribution and stronger notions of convergence.  In particular, assuming only a finite moment condition on the atom distribution, we can obtain convergence in the vague sense, and assuming an additional regularity condition, we can upgrade this convergence to locally $L^1$ convergence.

As an application, we determine the limiting distribution of the number of eigenvalues $N_I$ in a short interval 
$I$ of length $\Theta (1/n)$. As a corollary of this result, we obtain an extension of a result of Jimbo et. al. 
concerning the behavior of spacing in the bulk. 
\end{abstract}

\maketitle


\section{Introduction}

\subsection{Correlation functions}

This paper is concerned with the phenomenon of \emph{bulk universality} for the eigenvalue distribution of random Wigner ensembles.  To explain this phenomenon we need some notation.  Given a random Hermitian $n \times n$ matrix $M_n$, we can form the $n$ real eigenvalues (counting multiplicity), which we order as
$$ \lambda_1(M_n) \leq \ldots \ldots \le \lambda_n(M_n).$$
These $n$ random real variables can be viewed as describing a point process $\sigma(M_n) := \{ \lambda_1(M_n),\ldots,\lambda_n(M_n) \}$.  Associated to this point process, we can define the (unnormalized) \emph{$k$-point correlation functions} $R_n^{(k)}: \R^k \to \R^+$, which we can define by duality as the unique symmetric measurable function (or distribution) on $\R^k$ with the property that
\begin{equation}\label{rdef}
 \int_{\R^k} F(x_1,\ldots,x_k) R_n^{(k)}(x_1,\ldots,x_k)\ dx_1\ldots dx_k =
k! \E \sum_{1 \leq i_1 < \ldots < i_k \leq n} F( \lambda_{i_1}(M_n), \ldots, \lambda_{i_k}(M_n) )
\end{equation}
for any symmetric continuous, compactly supported function $F: \R^k \to \R$.  In particular, with this normalization, we have
$$ \int_{\R^k} R_n^{(k)}(x_1,\ldots,x_k)\ dx_1\ldots dx_k = \frac{n!}{(n-k)!}.$$
The $L^1$-normalized $k$-point correlation functions $p_n^{(k)} := \frac{(n-k)!}{n!} R_n^{(k)}$ are also used in the literature, but we will stick with the above normalization.

If $M_n$ is a discrete random matrix ensemble, then $R_n^{(k)}$ can only be defined in the sense of distributions.  But if $M_n$ is a continuous random matrix ensemble (with a continuous probability density function), then $R_n^{(k)}$ becomes a continuous symmetric function that vanishes whenever the $x_1,\ldots,x_k$ are not all distinct, and can be defined more explicitly for distinct reals $x_1,\ldots,x_k$ by the formula
$$ R_n^{(k)}(x_1,\ldots,x_k) = \lim_{\eps \to 0} \frac{1}{\eps^k} \P( E_{\eps,x_1,\ldots,x_k} )$$
where $E_{\eps,x_1,\ldots,x_k}$ is the event that there is an eigenvalue of $M_n$ in the interval $[x_i,x_i+\eps]$ for each $1 \leq i \leq k$.  Thus, for instance, the joint probability distribution of the ordered eigenvalues $\lambda_1 \leq \ldots \leq \lambda_n$ is given by the probability measure
$$ n! \rho_n(\lambda_1,\ldots,\lambda_n) 1_{\lambda_1 \leq \ldots \leq \lambda_n} d\lambda_1 \ldots d\lambda_n,$$
where $\rho_n := \frac{1}{n!} R_n^{(n)}$, and the remaining correlation functions $R_n^{(k)}$ can be derived from the probability density function $\rho_n$ by the integration formula
$$ R_n^{(k)}(x_1,\ldots,x_k) = \frac{n!}{(n-k)!} \int_{\R^{n-k}} \rho_n(x_1,\ldots,x_n)\ dx_{k+1} \ldots dx_n.$$
The $1$-point correlation function $R_n^{(1)}$ controls the density of states; indeed, for continuous random ensembles at least, one has the formula
$$ \E N_I = \int_I R_n^{(1)}(x)\ dx$$
for any interval $I$, where $N_I$ is the number of eigenvalues in $I$.

\section{Wigner matrices and the semicircle law}

We now restrict attention to a specific class of random matrix ensembles, namely the \emph{Wigner ensembles}.

\begin{definition}[Wigner matrices]\label{def:Wignermatrix}  Let $\xi$, $\tilde \xi$ be real random variables with mean zero and variance $1$,
and let $n \geq 1$ be an integer. An $n \times n$ \emph{Wigner Hermitian matrix} $M_n$ with atom distributions $\xi$, $\tilde \xi$ is defined to be a  random Hermitian $n \times n$ matrix $M_n$ with upper triangular complex entries $\frac{1}{\sqrt{n}} \zeta_{ij} :=  \frac{1}{\sqrt{n}} \frac{1}{ \sqrt 2} (\xi_{ij} + \sqrt{-1} \tau_{ij})$ ($1\le i < j \le n$) and diagonal real entries $\frac{1}{\sqrt{n}} \zeta_{ii}$
($1\le i \le n$) where the $\xi_{ij}, \tau_{ij}, \zeta_{ii}$ are jointly independent random variables, with $\xi_{ij}, \tau_{ij}$ having the distribution of $\xi$ for $1 \leq i < j \leq n$, and $\zeta_{ii}$ having the distribution of $\tilde \xi$ for $1 \leq i \leq n$.
\end{definition}

\begin{example} The famous \emph{Gaussian Unitary Ensemble} (GUE) is the special case of the Wigner ensemble in which the atom distributions $\xi, \tilde \xi$ are gaussian random variables, $\xi, \tilde \xi \equiv N(0,1)$.  At the opposite extreme, the \emph{complex Bernoulli ensemble} is an example of a discrete Wigner ensemble in which the atom distributions $\xi, \tilde \xi$ equal $+1$ with probability $1/2$ and $-1$ with probability $1/2$.
\end{example}

For these matrices, the bulk distribution of the eigenvalues is governed by the \emph{semicircular distribution} $\rho_{\operatorname{sc}}(x)\ dx$, where $\rho_{\operatorname{sc}}: \R \to \R$ is the function
$$ \rho_{\operatorname{sc}}(x) := \frac{1}{2\pi} (4-x^2)_+^{1/2}.$$
Indeed, we have

\begin{theorem}[Wigner semicircular law]  Let $M_n$ be a Wigner Hermitian matrix with fixed atom distributions $\xi,\tilde \xi$ (independent of $n$).  Then the normalized $1$-point correlation function $x \mapsto \frac{1}{n} R_n^{(1)}(x)$ converges weakly to $\rho_{\operatorname{sc}}(x)$ in the limit $n \to \infty$, thus
$$ \int_\R \frac{1}{n} R_n^{(1)}(x) F(x)\ dx \to \int_\R \rho_{\operatorname{sc}}(x) F(x)\ dx$$
for any continuous, compactly supported function $F: \R \to \R$.
\end{theorem}

\begin{proof} See \cite{Pas}.
\end{proof}

The semicircular law suggests that at any \emph{bulk energy level} $-2 < u < 2$, the mean eigenvalue spacing of the eigenvalues $\lambda_i(M_n)$ in the vicinity of $u$ should be close to $\frac{1}{n \rho_{\operatorname{sc}}(u)}$.  As such, it is natural to introduce the \emph{normalized $k$-point correlation functions} $\rho_{n,u}^{(k)}: \R^k \to \R^+$ for $1 \leq k \leq n$, localized to the energy level $u$, by the formula
\begin{equation}\label{rho-def}
 \rho_{n,u}^{(k)}( t_1,\ldots,t_k ) := \frac{1}{(n\rho_{\operatorname{sc}}(u))^k} R_n^{(k)}\left( u + \frac{t_1}{n \rho_{\operatorname{sc}}(u)}, \ldots, u + \frac{t_k}{n \rho_{\operatorname{sc}}(u)} \right).
 \end{equation}
(Again, when $M_n$ is a discrete ensemble, the $\rho_{n,u}^{(k)}$ should be interpreted as distributions or measures rather than as symmetric functions.)

Now fix $k \geq 1$, $-2 < u < 2$, and the atom distributions $\xi,\tilde \xi$, and consider the limiting behaviour of the normalised $k$-point correlation functions $\rho_{n,u}^{(k)}$ as $n \to \infty$.  A basic conjecture in the subject, due to Wigner, Dyson, and Mehta, can be stated informally as follows:

\begin{conjecture}[Wigner-Dyson-Mehta bulk universality conjecture, informal version] \cite{Meh,Erd} Fix $k \geq 1$, $-2 < u < 2$, and atom distributions $\xi, \tilde \xi$.  Then $\rho_{n,u}^{(k)} \to \rho_{\Dyson}^{(k)}$ as $n \to \infty$, where the \emph{Dyson sine kernel $k$-point correlation functions} $\rho_{\Dyson}^{(k)}: \R^k \to \R^+$ are defined by the formula
\begin{equation}\label{dyson-def}
 \rho_{\Dyson}^{(k)}(t_1,\ldots,t_k) := \det( K_\Dyson( t_i, t_j) )_{1 \leq i,j \leq k}
\end{equation}
and $K_\Dyson(t, t') := \frac{\sin(\pi(t'-t))}{\pi(t'-t)}$ is the \emph{Dyson sine kernel}.
\end{conjecture}

This conjecture is imprecise because the nature of the convergence of the function (or measure) $\rho_{n,u}^{(k)}$ to $\rho_\Dyson^{(k)}$ is not specified.  There are of course infinitely many modes of convergence that one could consider.  A   mode of convergence which has drawn considerable attention is vague convergence:

\begin{itemize} 

\item (Vague convergence) For any continuous, compactly supported function $F: \R^k \to \R^+$, one has
$$ \lim_{n \to \infty} \int_{\R^k} F(t) \rho_{n,u}^{(k)}(t)\ dt = \int_{\R^k} F(t) \rho_\Dyson^{(k)}(t)\ dt.$$

\end{itemize}

In recent years,  attention has also been paid to a weaker form of convergence:

\begin{itemize} 

\item (Average vague convergence) For any continuous, compactly supported function $F: \R^k \to \R^+$, one has
$$ \lim _{b \to 0^{+}} \lim_{n \to \infty} \frac{1}{2b} \int_{u-b}^{u+b}\int_{\R^k} F(t) \rho_{n,u}^{(k)}(t)\ dt = \int_{\R^k} F(t) \rho_\Dyson^{(k)}(t)\ dt.$$

\end{itemize}

The first objective of this paper is to provide an almost complete solution for vague convergence (Theorem \ref{main}). Namely, in this paper we establish the following theory, which asserts that universality holds in the vague convergence sense under the assumption that the atom variables have bounded high moments.

\begin{theorem}[Bulk universality for vague convergence]\label{main}  The Wigner-Dyson-Mehta conjecture for vague convergence is true whenever the atom distributions $\xi, \tilde \xi$ have finite $C_0^{th}$ moment, thus $\E |\xi|^{C_0}, \E |\tilde \xi|^{C_0} < \infty$, for some sufficiently large absolute constant $C_0$.
\end{theorem}

Under mild assumptions (such as convergence which is locally uniform in the $u$ parameter), vague convergence also implies averaged vague convergence.  However, this implication cannot be automatically reversed unless one has some uniform control on the regularity of the $\rho_{n,u}$ or $\rho_{n,u'}$, such as equicontinuity.  We also observe that for vague convergence that one can restrict attention without loss of generality to smooth compactly supported functions $F$, thanks to the Weierstrass approximation theorem.

Having established the Wigner-Dyson-Mehta conjecture in this form, we would like to raise new challenges  by considering other (stronger) modes of convergence. We focus on the following three modes (in decreasing order of  strength).

\begin{enumerate}
\item (Local uniform convergence) For every compact subset $K$ of $\R^k$, one has
$$ \lim_{n \to \infty} \sup_{t \in K} |\rho_{n,u}^{(k)}(t)-\rho_\Dyson^{(k)}(t)| = 0.$$
\item (Local $L^1$ convergence) For every compact subset $K$ of $\R^k$, one has
$$ \lim_{n \to \infty} \int_K |\rho_{n,u}^{(k)}(t)-\rho_\Dyson^{(k)}(t)|\ dt = 0.$$
\item (Weak convergence) For any $L^\infty$, compactly supported function $F: \R^k \to \R^+$, one has
$$
 \lim_{n \to \infty} \int_{\R^k} F(t) \rho_{n,u}^{(k)}(t)\ dt = \int_{\R^k} F(t) \rho_\Dyson^{(k)}(t)\ dt.
$$

\end{enumerate}

These three stronger notions of convergence, namely local uniform convergence, local $L^1$ convergence, and weak convergence, are only natural for continuous Wigner ensembles, since for discrete Wigner ensembles, $\rho_{n,u}^{(k)}$ is a discrete probability measure and is thus supported on a set of Lebesgue measure zero.  

We may now pose the question of determining the atom distributions $\xi, \tilde \xi$ for which the Wigner-Dyson-Mehta conjecture is true for local uniform convergence (resp. weak convergence).  By examining carefully and slightly modifying  an existing proof in current literature \cite{EPRSY} we note  the following partial result.

\begin{theorem}[Bulk universality for local $L^1$ convergence]\label{main2} For $k \geq 1$, there exists an integer $J_k \geq 1$ such that the following statements hold.  Assume that the atom distributions $\xi, \tilde \xi$ have a continuous distribution $e^{-V(x)} e^{-x^2}\ dx$, $e^{-\tilde V(x)} e^{-x^2}\ dx$ obeying the bounds
$$ e^{-V(x)} e^{-x^2} \leq C e^{-\delta x^2}$$
and
$$ |\frac{d^j}{dx^j} V(x)| \leq C (1+x^2)^m$$
for some $C,\delta,m>0$ and all $1 \leq j \leq J_k$.  Then the Wigner-Dyson-Mehta conjecture holds for these atom distributions and this value of $k$ in the local $L^1$ sense.

If $k=2$, then one can take $J_k=6$.
\end{theorem}

Thus, if one wants to obtain the Wigner-Dyson-Mehta conjecture in the local $L^1$ sense for a fixed value of $k$, one only needs a finite number of regularity hypotheses on the distribution; but if one wants to use this theorem to obtain local $L^1$ convergence for all $k$, one needs an infinite number of such hypotheses.  This is likely an artefact of the proof method, however, and one should be able to get local $L^1$ convergence for all $k$ assuming only a finite amount of regularity.

\subsection{Prior results} \label{section:prior}

Theorem \ref{main} is the last step of  a long series of results, which we now survey. 
We will focus exclusively on the bulk case $-2 < u < 2$; there are also several analogous results in the edge case $u = \pm 2$ (see \cite{sinai1}, \cite{sinai2}, \cite{Sos1}, \cite{TVedge}, \cite{Joh2}, \cite{EYY3}) which we do not discuss here.

The first results on Wigner-Dyson-Mehta conjecture were for the GUE ensemble.  In this case, one has the explicit \emph{Gaudin-Mehta formula}
$$ R_n^{(k)}(x_1,\ldots,x_k) = \det( K_n(x_i,x_j) )_{1\leq i,j \leq k}$$
for any $k \geq 1$, where the kernel $K_n(x,y)$ is given by the formula
$$ K_n(x,y) := \sqrt{n} \sum_{k=0}^{n-1} P_k(\sqrt{n} x) e^{-nx^2/2} P_k(\sqrt{n} y) e^{-ny^2/2}$$
and $P_0,\ldots,P_{n-1}$ are the first $n$ \emph{Hermite polynomials}, normalized so that each $P_i$ is of degree $i$, and are orthonormal with respect to the measure $e^{-x^2/2}\ dx$; see \cite{gin,Meh}.  We may renormalize the Gaudin-Mehta formula as
$$ \rho_{n,u}^{(k)}(t_1,\ldots,t_k) = \det( K_{n,u}(t_i,t_j) )_{1\leq i,j \leq k}$$
for any $k \geq 1$ and $-2 < u < 2$, where
$$ K_{n,u}( t, t' ) := \frac{1}{n \rho_{\operatorname{sc}}(u)} K_n\left( u + \frac{t}{n\rho_{\operatorname{sc}}(u)}, u + \frac{t'}{n \rho_{\operatorname{sc}}(u)} \right).$$
A standard calculation using the Plancherel-Rotarch asymptotics of Hermite polynomials and the Christoffel-Darboux formula (see \cite{Dys}) shows that $K_{n,u}$ converges locally uniformly to $K_\Dyson$ as $n \to \infty$, where $-2 < u < 2$ is fixed.  As such, the Wigner-Dyson-Mehta conjecture is true in the locally uniform sense, and thus also in the local $L^1$, weak, vague, and averaged vague senses; see \cite[Lemma 3.5.1]{AGZ} for a proof.

Analogous results are known for much wider classes of invariant random matrix ensembles, see e.g. \cite{DKMVZ}, \cite{PS}, \cite{BI}.  However, we will not discuss these results further here, as they do not directly impact on the case of Wigner ensembles.

Returning to the Wigner case, the next major breakthrough was by Johansson \cite{Joh1}, who considered atom distributions $\xi, \tilde \xi$ that were \emph{gauss divisible} in the sense that they could be expressed as
$$ \xi = e^{-t/2} \xi_1 + (1-e^{-t})^{1/2} \xi_2; \quad \tilde \xi = e^{-t/2} \tilde \xi_1 + (1-e^{-t})^{1/2} \tilde \xi_2,$$
where $t>0$, $\xi_1, \tilde \xi_1$ were real random variables of mean zero and variance one, and $\xi_2, \tilde \xi_2 \equiv N(0,1)$ were gaussian random variables that were independent of $\xi_1, \tilde \xi_1$ respectively.  Equivalently, the distributions of $\xi, \tilde \xi$ are obtained from those of $\xi_1, \tilde \xi_1$ by applying the Ornstein-Uhlenbeck process for time $t$, and the distributions of the Wigner matrix $M_n$ are similarly obtained from an initial Wigner ensemble $M_{n,1}$ by a matrix version of the Ornstein-Uhlenbeck process, with the eigenvalues then evolving by the process of \emph{Dyson Brownian motion}.  Note that gauss divisible distributions are automatically continuous (and even smooth), and so it makes sense to talk about convergence in the weak or locally uniform senses in this setting.  The result of \cite{Joh1} is then that if the Ornstein-Uhlenbeck time $t$ is fixed and positive independently of $n$, and if the factor distributions $\xi_1, \tilde \xi_1$ have sufficiently many moments finite, then the Wigner-Dyson-Mehta conjecture is true in the weak sense.  Informally, the results of \cite{Joh1} assert that the renormalized $k$-point correlation functions $\rho_{n,u}^{(k)}$ converge to equilibrium by time $t$ for any $t>0$ independent of $n$.  The main tool used in \cite{Joh1} was an explicit determinantal formula for the correlation functions in the gauss divisible case, essentially due to Br\'ezin and Hikami \cite{brezin}.

In Johansson's result, the time parameter $t > 0$ had to be independent of $n$.  It was realized by Erd\H{o}s, Ramirez, Schlein, and Yau that one could obtain many further cases of the Wigner-Dyson-Mehta conjecture if one could extend Johansson's result to much shorter times $t$ that decayed at a polynomial rate in $n$.  This was first achieved (again in the context of weak convergence) for $t > n^{-3/4+\eps}$ for an arbitrary fixed $\eps>0$ in \cite{ERSY}, and then to the essentially optimal case $t > n^{-1+\eps}$ (for weak convergence) in \cite{EPRSY}, assuming that the atom distributions $\xi, \tilde \xi$ were continuous whose distribution was sufficiently smooth distribution (e.g. when $k=2$ one needs a $C^6$ type condition), and also decayed exponentially.  The methods used in \cite{EPRSY} were an extension of those in \cite{Joh1}, combined with an approximation argument (the ``method of time reversal'') that approximated a continuous distribution by a gauss divisible one (with a small value of $t$); the arguments in \cite{ERSY} are based instead on an analysis of the Dyson Brownian motion.

Simultaneously with results in \cite{EPRSY}, the authors in \cite{TV} introduced a different approach, based on what we call the \emph{Four Moment Theorem}, that allowed one to extend claims such as those in the Wigner-Mehta-Dyson conjecture (in the context of vague convergence) from one choice of atom distributions to another, provided that the first four moments in the atom distribution $\xi$ were matching (or close to matching), and provided that the atom distributions $\xi, \tilde \xi$ obeyed an exponential decay condition, namely that
\begin{equation}\label{exp-decay}
\P( |\xi| \geq t ), \P( |\tilde \xi| \geq t ) \leq C t^{-c} 
\end{equation}
for all $t>0$ and some constants $C,c>0$.  However, the distributions $\xi,\tilde \xi$ were not required to be continuous, and in particular could be discrete.  By combining the Four Moment Theorem with the results of Johansson (which provided a rich class of comparison distributions with which to match moments), the Wigner-Mehta-Dyson conjecture
for vague convergence was then established in \cite{TV} for atom distributions $\xi$ that were supported on at least three points and whose third moment $\E \xi^3$ vanished, assuming the exponential decay condition \eqref{exp-decay} on $\xi$ and $\tilde \xi$.  The vanishing third moment condition was needed in order to obtain a strong localization result for the individual eigenvalues $\lambda_i(M_n)$ of the Wigner matrix $M_n$.  The three point condition is a technical condition needed to in order to match moments with a better behaved ensemble, but is quite mild, as it effectively only excludes the case of Bernoulli random ensembles.

Shortly afterwards, it was realized in \cite{ERSTVY} that the methods in \cite{EPRSY} and \cite{TV} could be combined to handle a wider class of atom distributions, but at the cost of weakening vague convergence to averaged vague convergence.  Specifically, in \cite{ERSTVY} the Wigner-Mehta-Dyson conjecture for averaged vague convergence was established for atom distributions $\xi, \xi'$ that were assumed to have an exponential decay condition \eqref{exp-decay} , but for which no regularity, support, or moment conditions were imposed.  The need to retreat to averaged vague convergence was again due to the lack (at the time) of a strong localization result of individual eigenvalues for this ensemble; in particular, as remarked in \cite{ERSTVY} one could upgrade averaged vague convergence to vague convergence if one re-imposed the vanishing third moment condition.

Next, in \cite{ESY}, the method of \emph{local relaxation flow} was introduced, which provided a new way to analyze Dyson Brownian motion on short time scales that did not rely on explicit formulae of Br\'ezin-Hikami type.  This gave a simpler and more general approach to universality, but for technical reasons it relied more heavily on the ability to average in the energy parameter $u$, and so could only give progress on the Wigner-Mehta-Dyson conjecture in the context of averaged vague convergence.  In particular, in \cite{ESY} an alternate proof of the Wigner-Mehta-Dyson conjecture for averaged vague convergence was established\footnote{The main theorem in \cite{ESY} assumes a log-Sobolev condition on the atom distribution, but it is remarked in that paper that this condition can be removed assuming the three points condition, thanks to the Four Moment Theorem.} for atom distributions $\xi$ supported on at least three points and with $\xi,\tilde \xi$ obeying an exponential decay condition \eqref{exp-decay} (with the method extending for the first time to real symmetric or symplectic ensembles as well), with the three point condition then being removed subsequently in \cite{EYY2}.  These results were generalized to covariance matrices in \cite{BenP}, \cite{ESYY} and to generalised Wigner ensembles (in which the entries need not be iid and are allowed some fluctuation in their variances) in \cite{EYY}, \cite{EYY2}.

In \cite{TVscv} it was observed that the exponential decay condition \eqref{exp-decay} on the atom distributions in the Four Moment Theorem could be relaxed to a finite $C_0^{th}$ moment condition for some sufficiently large absolute constant $C_0$ (e.g. $C_0=10^4$ would suffice).  As a consequence, in many of the preceding results on the Wigner-Dyson-Mehta conjecture, the exponential decay hypothesis \eqref{exp-decay} could be replaced with a finite moment hypothesis.

In the very recent paper \cite{maltsev} (see in particular Corollary 1.3 of that paper), the Wigner-Dyson-Mehta conjecture for the strongest notion of convergence, namely local uniform convergence, was established in the $k=1$ case, assuming a sufficient number of smoothness and moment conditions on the atom distribution.

For more detailed discussion of the above results, see the surveys \cite{Erd}, \cite{Alice}.

\subsection{Application to the counting function}

Theorem \ref{main} implies in particular a moment bound for the counting function
$$ N_I := \# \{ 1 \leq i \leq n: \lambda_i(M_n) \in I \}$$
on intervals in the bulk at the scale of the mean eigenvalue spacing:

\begin{corollary}\label{nik}  Suppose the atom distributions obey the bounds $\E |\xi|^{C_0}, \E |\tilde \xi|^{C_0} < C_1$ for some sufficiently large $C_0>0$ and some $C_1>0$.  Let $\eps>0$, and let $I \subset [-2+\eps,2-\eps]$ be an interval of length at most $K/n$ for some $K>0$.  Then for any $k \geq 1$, we have\footnote{Here we use the asymptotic notation $X \ll Y$ or $X=O(Y)$ if $|X| \leq CY$ for some constant $C$ depending on the indicated parameters.} 
\begin{equation}\label{wek}
 \E N_I^k \ll 1
\end{equation}
where the implied constant depends only on $C_0, C_1,\eps,K,k$.
\end{corollary}

Such an estimate was previously established in \cite{ESY-wegner} under an additional regularity hypothesis on the atom distribution, as well as a stronger (subgaussian) decay hypothesis.  

\begin{proof}  Fix $C_0,C_1,\eps,K,k$, and allow all implied constants to depend on these parameters.  In view of the trivial bound $N_I \leq n$ we may assume that $n$ is sufficiently large depending on the fixed parameters.  We may then take $I = [u-K/2n,u+K/2n]$ for some $-2+\eps/2 < u < 2-\eps/2$.  

It suffices to show that
$$ \E \binom{N_I}{k} \ll 1.$$
Using \eqref{rdef}, \eqref{rho-def}, we can bound
$$ \E \binom{N_I}{k} \leq \frac{1}{k!} \int_{\R^k} F(t) \rho_n^{(k)}(t)\ dt$$
where $F: \R^k \to \R^+$ is a smooth, compactly supported function that equals $1$ on $[-K/2,K/2]^k$.  From Theorem \ref{main} we have
$$ |\int_{\R^k} F(t) \rho_n^{(k)}(t)\ dt - \int_{\R^k} F(t) \rho_\Dyson^{(k)}(t)\ dt| \to 0$$
as $n \to \infty$, and so the left-hand side is bounded in $n$.  An inspection of the proof of Theorem \ref{main} shows that this bound only depends\footnote{For the purposes of establishing Theorem \ref{ni-asym} below, this more refined version of Theorem \ref{main} is not necessary, as one can use the weaker conclusion $\limsup_{n \to \infty} \E N_I^k < \infty$ instead as a substitute for \eqref{wek}.} on the quantities $C_0, C_1, \eps, K, k$ (i.e. it is uniform in $u$ and in $\xi, \tilde \xi$ once the quantities $C_0,C_1,\eps,K,k$ are fixed); thus
$$ |\int_{\R^k} F(t) \rho_n^{(k)}(t)\ dt - \int_{\R^k} F(t) \rho_\Dyson^{(k)}(t)\ dt| \ll 1.$$
As $\rho_\Dyson$ is bounded, the claim follows.
\end{proof}

With a little more effort, one can obtain the asymptotic law for $N_I$ in the case when $|I|$ is comparable to $1/n$:

\begin{theorem}[Asymptotic for $N_I$]\label{ni-asym}  Suppose the atom distributions obey the bounds $\E |\xi|^{C_0}, \E |\tilde \xi|^{C_0} < C_1$ for some sufficiently large $C_0>0$ and some $C_1>0$.  Let $\eps>0$ and $K > 0$ be independent of $n$.  For any $n$, let $u = u_n$ be an element of $[-2+\eps,2-\eps]$, and let $I = I_n$ be the interval $I := [u, u + \frac{K}{\rho_{\operatorname{sc}}(u) n}]$.  Then $N_I$ converges in distribution to the random variable $\sum_{j=1}^\infty \xi_j$, where $\xi_j$ are independent Bernoulli indicator random variables with expectation $\E \xi_j = p_j$, where $p_1 \geq p_2 \geq \ldots \geq 0$ are the eigenvalues of the (compact, positive semi-definite) integral operator $T f(x) := \int_{[0,K]} K_\Dyson(x,y) f(y)\ dy$ on $L^2([0,K])$.  In particular, the probability $\P(N_I=0)$ that $I$ has no eigenvalues is equal to $\prod_{j=1}^\infty (1-p_j) + o(1)$.
\end{theorem} 

This result is well known for GUE, thanks to the theory of determinantal processes, but the extension to arbitrary Wigner matrices (with the finite moment condition) is new.  We prove this result in Section \ref{asym-sec}. 

By definition, the quantity $\prod_{j=1}^\infty (1-p_j)$ appearing in Theorem \ref{ni-asym} is equal to the Fredholm determinant $\det(1-T)$.  This determinant was computed by Jimbo, Miwa, Mori and Sato (see \cite{JMM} or \cite[Theorem 3.1.2]{AGZ}) as the solution to a certain ODE in the length parameter $K$.  As a consequence, we have

\begin{corollary}\label{lsv-gue} With the notation and assumptions of Theorem \ref{ni-asym}, one has
$$ \P( N_I = 0  ) \to \exp( \int_0^K  \frac{f(x)}{x} dx )$$
as $n \to \infty$, where $f:\R \to \R$ is the solution of the differential equation
$$(Kf'')^2 + 4 (Kf'-f) (Kf' -f + (f')^2) =0 $$
with the asymptotics  $f(K)  = \frac{-K}{\pi} -\frac{K^2}{\pi^2} -\frac{K^3}{\pi^3} + O(K^4)$ as $K \to 0$.
\end{corollary}

It is likely that one can also obtain asymptotic distributions for $N_I$ for longer intervals $I$, as in \cite{dall}, but we do not pursue this issue here.

We thank the anonymous referees for many useful corrections and suggestions.

\section{Proof of Theorem \ref{main}}\label{main-pf}

In this section we prove Theorem \ref{main}.  This proof is a simple combination of existing arguments and results which have been obtained in the last few years. 
The core of our argument is the following. In \cite{TV}, the authors worked out a method to prove universality using the Four moment theorem combined with Johansson's theorem. A finer version of this theorem from \cite{ERSTVY} enables one to combine an approximate version of the Four moment theorem with a version of Johansson's theorem where the time parameter $t$ tends to zero with $n$. This gives universality under the assumption that the third moment vanishes. The extra observation here is that we can omit this assumption using a recent rigidity result from \cite{EYY3}.  Details now follow. 

We need the following technical definition.

\begin{definition}[Asymptotic moment matching]\label{def:asymmatch} Let $\delta^{(k)}= (\delta_1, \ldots, \delta_k)$ be a sequence of $k$ positive numbers. 
We say that two complex random variables $\zeta$ and $\zeta'$ \emph{$\delta^{(k)}$-match to order} $k$ if
$$\Big|  \E \Re(\zeta)^m \Im(\zeta)^l - \E \Re(\zeta')^m \Im(\zeta')^l \Big| \le \delta_{m+l} $$
for all $m, l \ge 0$ such that $m+l  \le k$. 
\end{definition}

Set 
$$ \delta^{(4)}   := (0,0, n^{- 1/2 -c} , n^{-1/2-c})$$ where $c$ is a positive constant. (The first two coordinates are 0 as we would like to keep the mean $0$ and variance $1$ in all models.)

The approximate version of the Four moment theorem is the following (implicit in \cite{ERSTVY} and can be deduced easily from the proof of the Four moment theorem):

\begin{theorem}[Asymptotic Four Moment Theorem]\label{theorem:AppFour} 
There is a small absolute constant $c_0 > 0$ such that for integer  $k \geq 1$ the following holds.
 Let $M_n = \frac{1}{\sqrt{n}} (\zeta_{ij})_{1 \leq i,j \leq n}$ and $M'_n = \frac{1}{\sqrt{n}} (\zeta'_{ij})_{1 \leq i,j \leq n}$ be
 two Wigher Hermitian matrices where the atom distributions have finite $C_0^{\th}$ moment for some sufficiently large $C_0$. Assume furthermore that for any $1 \le  i<j \le n$, $\zeta_{ij}$ and
 $\zeta'_{ij}$  $\delta^{(4)}$-match to order $4$
  and for any $1 \le i \le n$, $\zeta_{ii}$ and $\zeta'_{ii}$ match  to order $2$.  Set $A_n := n M_n$ and $A'_n := n M'_n$,
  and let $G: \R^k \to \R$ be a smooth function obeying the derivative bounds
\begin{equation}\label{G-deriv}
|\nabla^j G(x)| \leq n^{c_0}
\end{equation}
for all $0 \leq j \leq 5$ and $x \in \R^k$.
 Then for any $1 \le i_1 < i_2 \dots < i_k \le n$, and for $n$ sufficiently large we have
\begin{equation} \label{eqn:approximation}
 |\E ( G(\lambda_{i_1}(A_n), \dots, \lambda_{i_k}(A_n))) -
 \E ( G(\lambda_{i_1}(A'_n), \dots, \lambda_{i_k}(A'_n)))| \le n^{-c_0}.
\end{equation}
\end{theorem}

\begin{remark}\label{extend} In the version of this theorem in \cite{TV}, \cite{ERSTVY} one required the atom distributions of both $M_n$ and $M'_n$ to obey an exponential decay condition \eqref{exp-decay}, rather than merely having a finite $C_0^{\th}$ moment.  However, the only reason why this exponential decay was ne	eded was to obtain a lower tail estimate for eigenvalue gaps (see \cite[Theorem 19]{TV}).  However, it was subsequently observed in \cite[Remark 38]{TVscv} that (by using a certain truncated version of the Four Moment Theorem) one could extend this lower tail estimate to the case of Wigner matrices whose atom distribution had finite $C_0^{\th}$ moment, and so Theorem \ref{theorem:AppFour} can also be extended to this regime.
\end{remark}

We fix $\xi, \tilde \xi, u$ as in that theorem, fix a $k \geq 1$, and fix a continuous compactly supported function $F: \R^k \to \R$; we will also need
a small constant $\eps > 0$ depending on $k$ ($\eps := \frac{1}{100k}$ will suffice).  We allow all implied constants in asymptotic notation to depend on these quantities.  It then suffices to show that
\begin{equation}\label{rhodys}
 \int_{\R^k} F \rho_{n,u}^{(k)}(t_1,\ldots,t_k)\ dt_1 \ldots dt_k = \int_{\R^k} F \rho_\Dyson^{(k)}(t_1,\ldots,t_k)\ dt_1 \ldots t_k + o(1).
 \end{equation}

Using the Stone-Weierstrass theorem to approximate a continuous compactly supported function uniformly by a smooth function of uniformly bounded support, we may assume without loss of generality that $F$ is smooth.  (The error in doing so can be upper bounded by a further application of \eqref{rhodys} applied to a slightly wider function $F$, taking advantage of the local integrability of $\rho_\Dyson^{(k)}$.)

We may assume that $n$ is sufficiently large depending on these quantities.  We also need a small absolute constant $\eps > 0$ ($\eps := 10^{-2}$ will suffice).  The distribution $\xi$ need not be bounded, but it is easy to see (e.g. using \cite[Lemma 28]{TV}) that there exists a bounded distribution $\xi'$ which matches moments with $\xi$ to fourth order in the sense that $\E \xi^i = \E (\xi')^i$ for $i=1,2,3,4$.  Next, we set $t := n^{-1+\eps}$ and introduce the modified atom distribution $\xi''$ defined by the formula
$$ \xi'' := e^{-t/2} \xi' + (1-e^{-t})^{1/2} g$$
where $g \equiv N(0,1)$ is independent of $\xi'$.  A routine computation shows that
\begin{equation}\label{xi-xip}
\E (\xi'')^i = \E \xi^i
\end{equation}
for $i=1,2$ and
\begin{equation}\label{xi-xip2}
 \E (\xi'')^i = \E \xi^i + O( n^{-1+\eps} )
\end{equation}
for $i=3,4$; thus $\xi$ and $\xi''$ have approximately matching moments to fourth order.  We define $\tilde \xi'$ and $\tilde \xi''$ from $\tilde \xi$ analogously.

Let $M''_n$ be the random matrix ensemble defined similarly to $M_n$, but with the atom distributions $\xi$, $\tilde \xi$ replaced by $\xi'', \tilde \xi''$.  This is another Wigner ensemble whose atom distribution is now gauss divisible with time parameter $t$, and which also obeys an exponential decay condition \eqref{exp-decay} (being the sum of a bounded distribution and a gaussian distribution).  As such, one can invoke \cite[Proposition 4]{ERSTVY} and conclude that $M''_n$ already obeys the required conclusion, thus
$$ \int_{\R^k} F (\rho_{n,u}^{(k)})''(t_1,\ldots,t_k)\ dt_1 \ldots dt_k = \int_{\R^k} F \rho_\Dyson^{(k)}(t_1,\ldots,t_k)\ dt_1 \ldots t_k + o(1),$$
where $(\rho_{n,u}^{(k)})''$ is defined analogously to $\rho_{n,u}^{(k)}$ but with $M_n$ replaced by $M''_n$.  It thus suffices to show that
$$ \int_{\R^k} F (\rho_{n,u}^{(k)})''(t_1,\ldots,t_k)\ dt_1 \ldots dt_k = \int_{\R^k} F \rho_{n,u}^{(k)}(t_1,\ldots,t_k)\ dt_1 \ldots t_k + o(1),$$
or equivalently (after a rescaling) that
\begin{align*}
& \sum_{1 \leq i_1 < \ldots < i_k \leq n}
\E G( n \lambda_{i_1}(M_n), \ldots, n \lambda_{i_k}(M_n) )\\
&\quad =
 \sum_{1 \leq i_1 < \ldots < i_k \leq n}
\E G( n \lambda_{i_1}(M''_n), \ldots, n \lambda_{i_k}(M''_n) ) + o(1)
\end{align*}
where $G: \R^k \to \R$ is the function
$$ G( t_1,\ldots,t_k ) := F( \rho_{\operatorname{sc}}(u) (t_1-nu), \ldots, \rho_{\operatorname{sc}}(u) (t_k-nu) ).$$
Note that the expression $G( \sqrt{n} \lambda_{i_1}(M_n), \ldots, \sqrt{n} \lambda_{i_k}(M_n) )$ is only non-zero when we have
\begin{equation}\label{lim}
\lambda_{i_1}(M_n),\ldots,\lambda_{i_k}(M_n) = u + O\left( \frac{1}{n} \right).
\end{equation}

Using the crude upper bound from \cite[Proposition 66]{TV}, we see that with probability $1-o(n^{-k})$, the number of eigenvalues $\lambda_i(M_n)$ or $\lambda_i(M'_n)$ that lie in the range \eqref{lim} is $O(n^{o(1)})$.  The exceptional event of probability $o(n^{-k})$ contributes at most $o(1)$ to the expression to be estimated and can thus be ignored.  In the remaining event, we see that the sum $\sum_{1 \leq i_1 < \ldots < i_k \leq n} G( n \lambda_{i_1}(M_n), \ldots, n \lambda_{i_k}(M_n) )$ is at most $O(n^{o(1)})$, and similarly for $M'_n$.  Thus we may in fact discard any event of probability $O(n^{-c})$ for any $c>0$.

We now need the following rigidity of eigenvalues theorem:

\begin{theorem}[Rigidity of eigenvalues]  If $0 < \eps, \kappa < 1$ are independent of $n$, and $-2+\kappa < u < 2-\kappa$, then for sufficiently large $n$, one has
$$ N_{[-2,u]}(M_n) = \int_{-2}^u \rho_{\operatorname{sc}}(x)\ dx + O(n^{\eps})$$
with probability $1-O(n^{-c})$ for some absolute constant $c>0$.
\end{theorem}

\begin{proof}  If the Wigner matrix has exponential decay, then this follows from \cite[Theorem 2.2]{EYY3} or \cite[Theorem 6.3]{EYY2} (and in this case one obtains a much higher probability of success, in particular obtaining $1-O(n^{-A})$ for any $A$).  The general case then follows from the Four Moment Theorem.
\end{proof}

Applying this theorem (with $u$ replaced by $u \pm n^{-1+2\eps}$, say), we see that with probability $1-O(n^{-c})$, the event \eqref{lim} only occurs when
\begin{equation}\label{reg}
 i_1,\ldots,i_k = n \int_{-2}^u \rho_{\operatorname{sc}}(x)\ dx + O(n^\eps).
\end{equation}
By the preceding discussion we may discard the exceptional event of probability $O(n^{-c})$ in which the above assertion fails.  We thus see that up to errors of $o(1)$, we may localize all the indices $i_1,\ldots,i_k$ to the regime \eqref{reg}.  By the triangle inequality, it thus suffices to show that
$$ \E G( \sqrt{n} \lambda_{i_1}(M_n), \ldots, \sqrt{n} \lambda_{i_k}(M_n) )
=
\E G( \sqrt{n} \lambda_{i_1}(M''_n), \ldots, \sqrt{n} \lambda_{i_k}(M''_n) ) + O(n^{-c_0})$$
for all $1 \leq i_1 \leq \ldots \leq i_k \leq n$ and some absolute constant $c_0>0$ independent of $\eps$.  But this follows from Theorem \ref{theorem:AppFour}.

\section{Proof of Theorem \ref{main2} }

The proof of this theorem is based on a   careful inspection of the proof of the weak convergence result in \cite{EPRSY}.
  From \cite[Theorem 1.1]{EPRSY} (in the $k=2$ case) and \cite[Remark 1.1]{EPRSY} (for the general $k$ case) we obtain the weak convergence for all functions $F$ which are bounded and of compact support; thus we have
\begin{equation}\label{lo}
 |\int_{\R^k} F(t) \rho_{n,u}^{(k)}(t)\ dt - \int_{\R^k} F(t) \rho_\Dyson^{(k)}(t)\ dt| \leq o(1)
\end{equation}
whenever $F$ is bounded and supported in a compact set $K$.

Let us now inspect the bounds on the convergence rate $o(1)$ in \eqref{lo} that come from the argument in \cite[\S 4]{EPRSY}.  That argument controls the left-hand side of \eqref{lo} by two expressions, denoted $(I)$ and $(II)$ in \cite[\S 4]{EPRSY}.  The first error term $(I)$ is shown to decay exponentially in $n$, and the dependence on $F$ only appears through a factor of $\|F\|_{L^\infty}$.  The error term $(II)$ is bounded using \cite[Proposition 3.1]{EPRSY}, which in our notation is an estimate of the form
\begin{equation}\label{hi}
 |\int_{\R^k} F(t) (\rho_{n,u}^{(k)})'(t)\ dt - \int_{\R^k} F(t) \rho_\Dyson^{(k)}(t)\ dt| \leq o(1)
\end{equation}
where $(\rho_{n,u}^{(k)})'$ is the $k$-point correlation function corresponding to a certain gauss divisible Wigner matrix.

The bound \eqref{hi} in turn proven using \cite[Proposition 3.3]{EPRSY}.  The convergence of the $k$-point correlation function provided by \cite[Proposition 3.3]{EPRSY} is uniform on bounded sets.  If one uses this uniform convergence in the argument used to prove \cite[Proposition 3.1]{EPRSY}, we see that the upper bound on \eqref{hi} is actually of the form $c(n) \|F\|_\infty$, where $c(n) \to 0$ as $n \to \infty$ depends on the support $K$ of $F$, but is otherwise independent of $F$.  Returning to the estimation of the term $(II)$ in \cite[\S 4]{EPRSY}, we conclude a similar bound for $(II)$. Putting all this together, we obtain a bound of the form
$$
 |\int_{\R^k} F(t) \rho_{n,u}^{(k)}(t)\ dt - \int_{\R^k} F(t) \rho_\Dyson^{(k)}(t)\ dt| \leq c'(n) \|F\|_\infty
$$
where $c'(n) \to 0$ as $n \to \infty$ depends on $K$ but is otherwise independent of $F$.  By duality, this implies that
$$ \int_K | \rho_{n,u}^{(k)}(t)\ dt - \rho_\Dyson^{(k)}(t)|\ dt \leq c'(n)$$
which gives the local $L^1$ convergence.

\begin{remark} A similar inspection the proof of \cite[Theorem 1.2]{Joh1} (and \cite[Lemma 3.1]{Joh1}) reveals that the weak convergence result in \cite{Joh1} could in fact also be retroactively upgraded to local $L^1$ convergence in a similar fashion.  It is likely that one can reduce the regularity and decay hypotheses on the above theorem, for instance by using the methods indicated in \cite[Section 5]{EPRSY}.  However, some minimal regularity hypothesis is certainly needed, as local $L^1$ convergence is of course not possible in the case of discrete distributions.  It is also likely that one can upgrade local $L^1$ convergence further to local uniform convergence under a sufficiently strong regularity hypothesis, especially in view of \cite[Proposition 3.3]{EPRSY} (and also \cite[Corollary 1.3]{maltsev} for the $k=1$ case).
\end{remark}





\section{Proof of Theorem \ref{ni-asym}}\label{asym-sec}

We now prove Theorem \ref{ni-asym}.  Fix $\eps, K, C_0, C_1$; we allow all implied constants to depend on these quantities.  From the trace formula
$$ \sum_{j=1}^\infty p_j = \int_{[0,K]} K_\Dyson(x,x)\ dx = K$$
we see that the $p_j$ are absolutely summable.  In fact, we have a stronger decay property:

\begin{lemma}[Decay faster than any polynomial]\label{so}  Let $A > 0$.  Then one has $p_j \ll_A j^{-A}$ for all $j$.
\end{lemma}

\begin{proof}  The Dyson kernel $K_\Dyson$ is smooth on $[0,K] \times [0,K]$, and can thus be smoothly extended to a function on the torus $(\R/2K\Z) \times (\R/2K\Z)$.  By a Fourier expansion, one can thus approximate $K_\Dyson(x,y)$ uniformly to error $O_A( M^{-2A} )$ by a Fourier series
$$ \sum_{a=-M}^M \sum_{b=-M}^M c_{a,b} e^{2\pi i a x / 2K} e^{2\pi i by / 2K},$$
for any positive integer $M$; thus
$$ K_\Dyson(x,y) = \sum_{a=-M}^M \sum_{b=-M}^M c_{a,b} e^{2\pi i a x / 2K} e^{2\pi i by / 2K} + O_A(M^{-2A})$$
for all $x, y \in [0,K]$.  This decomposition of the kernel $K_\Dyson$ induces a corresponding decomposition of the integral operator 
$$ Tf(x) = \int_{[0,K]} K_\Dyson(x,y) f(y)\ dy$$
as $T = T_1 + T_2$, where
$$ T_1 f(x) = \sum_{a=-M}^M \sum_{b=-M}^M c_{a,b} e^{2\pi i a x/2K} \int_{[0,K]} e^{2\pi i by/2K} f(y)\ dy$$
and $T_2$ is an operator with operator norm $O_A(M^{-2A})$.

Observe that $T_1$ is a finite rank operator, with rank at most $(2M+1)^2$.  By the Courant-Fisher min-max theorem, we thus see that apart from the $(2M+1)^2$ largest eigenvalues, all other eigenvalues of $T$ are of size $O_A(M^{-2A})$.  In other words, $p_j = O_A(M^{-2A})$ whenever $j > (2M+1)^2$, and the claim follows.
\end{proof}

This implies a subgaussian bound on $\sum_{j=1}^\infty \xi_j$:

\begin{lemma}  For any $\lambda > 0$, one has
$$ \P( \sum_{j=1}^\infty \xi_j > \lambda) \ll e^{-c \lambda^2}$$
for some constant $c > 0$ independent of $n$.
\end{lemma}

\begin{proof}  We may assume without loss of generality that $\lambda > 1$.  Because each $\xi_j$ is bounded by $1$, we have
$$
\P( \sum_{j=1}^\infty \xi_j > \lambda) \leq \P( \sum_{j > \lambda/2} \xi_j > \lambda/2).$$
From Lemma \ref{so}, the random variable $\sum_{j > \lambda/2} \xi_j$ has mean and variance $O_A( \lambda^{-A} )$ for any $A > 0$.  The claim then follows from the Chernoff inequality.
\end{proof}

As a consequence of the above lemma and the Carleman theorem (see e.g. \cite{bai}), the distribution of $\sum_{j=1}^\infty \xi_j$ is determined uniquely by its moments $\E (\sum_{j=1}^\infty \xi_j)^k$ for $k=1,2,\ldots$.  To prove Theorem \ref{ni-asym}, it thus suffices (by Prokhorov's theorem) to show that
$$\lim_{n\to \infty} \E N_I^k = \E \left(\sum_{j=1}^\infty \xi_j\right)^k$$
for $k=1,2,\ldots$.  As the monomial $n^k$ can be expressed as a linear combination of the binomial coefficients $\binom{n}{j}$ for $j=1,\ldots,k$, it suffices to show that
$$\lim_{n\to \infty} \E \binom{N_I}{k} = \E \binom{\sum_{j=1}^\infty \xi_j}{k}$$
for $k=1,2,\ldots$.  

By \eqref{rdef}, \eqref{rho-def}, one has
$$ \E \binom{N_I}{k} \leq \frac{1}{k!} \int_{\R^k} F(t) \rho_n(t)\ dt$$
whenever $F$ is a continuous compactly supported function with $F \geq 1_{[0,K]^k}$ pointwise, and similarly
$$ \E \binom{N_I}{k} \geq \frac{1}{k!} \int_{\R^k} F(t) \rho_n^{(k)}(t)\ dt$$
whenever $F$ is a continuous compactly supported function with $F \leq 1_{[0,K]^k}$ pointwise.  Taking limits using Theorem \ref{main}, and then taking advantage of the bounded nature of $\rho_\Dyson^{(k)}$ (and Urysohn's lemma), we conclude that
$$ \lim_{n \to \infty} \E \binom{N_I}{k} = \frac{1}{k!} \int_{[0,K]^k} \rho_\Dyson^{(k)}(t)\ dt.$$
Meanwhile we have\footnote{Note that all formal interchanges of summation or integration in this argument can be easily justified using Lemma \ref{so}.}
$$ \binom{\sum_{j=1}^\infty \xi_j}{k} = \frac{1}{k!} \sum_{j_1,\ldots,j_k \geq 1, \hbox{ distinct}} p_{j_1} \ldots p_{j_k}$$
so it suffices to show that
$$ \int_{[0,K]^k} \rho_\Dyson^{(k)}(t)\ dt = \sum_{j_1,\ldots,j_k \geq 1, \hbox{ distinct}} p_{j_1} \ldots p_{j_k}.$$
By the spectral theorem, we may write
$$ K_\Dyson(x,y) = \sum_{j=1}^\infty p_j \phi_j(x) \phi_j(y)$$
for some orthonormal real sequence $\phi_j \in L^2([0,K])$, and so (by \eqref{dyson-def} and the multilinearity of determinant)
$$ \rho_\Dyson^{(k)}(t_1,\ldots,t_k) = \sum_{j_1,\ldots,j_k \geq 1} p_{j_1} \ldots p_{j_k} \det( \phi_{j_a}(t_a) \phi_{j_a}(t_b) )_{1 \leq a,b \leq k}.$$
It thus suffices to show that the integral
\begin{equation}\label{jo}
 \int_{[0,K]^k} \det( \phi_{j_a}(t_a) \phi_{j_a}(t_b) )_{1 \leq a,b \leq k}\ dt_1 \ldots dt_k
\end{equation}
equals $1$ when the $j_1,\ldots,j_k$ are distinct, and vanishes otherwise.

If two of the $j_a$ are equal, then we see that two of the rows in the determinant are linearly dependent, so \eqref{jo} indeed vanishes.  Now suppose that the $j_1,\ldots,j_k$ are all distinct.  By cofactor expansion, we may then write \eqref{jo} as
$$
\sum_{\sigma \in S_k} (-1)^{\operatorname{sgn}(\sigma)} \int_{[0,K]^k} \prod_{b=1}^k \phi_{j_b}(t_b) \phi_{j_{\sigma(b)}}(t_b)\ dt_b.$$
By the orthonormality of the $\phi_j$ and Fubini's theorem, the integral here equals $1$ when $\sigma$ is the identity permutation and vanishes otherwise.  The claim follows.

\begin{remark} An inspection of the above argument reveals that one can replace the interval $[0,K]$ by an arbitrary compact set $A$, with the interval $I$ then being replaced by the set $\{ u + \frac{t}{\rho_{\operatorname{sc}}(u) n}: t \in A\}$.  We leave the details to the interested reader.
\end{remark}

\end{document}